\documentclass[preprint,12pt]{elsarticle}

\usepackage{amscd}
\usepackage{amsmath}
\usepackage{amsthm}
\usepackage{amsfonts}
\usepackage{graphicx}
\usepackage{amssymb}
\usepackage{color}
\usepackage{amsbsy}
\usepackage{graphicx}
\usepackage{amssymb,color,amsbsy}%, amsthm}
\usepackage{mathtools}
\usepackage{leftindex}

\usepackage{float}
\usepackage{comment}

\usepackage{stmaryrd}

\usepackage[ruled]{algorithm2e}

\usepackage[matrix,arrow]{xy}
\DeclareMathAlphabet{\mathpzc}{OT1}{pzc}{m}{it}

\usepackage{pgfpages}
\usepackage{tikz}
\usetikzlibrary{shapes.geometric}
\usetikzlibrary{decorations.pathmorphing}
\usetikzlibrary{arrows}
\usetikzlibrary{backgrounds}
\usetikzlibrary{positioning}
\usetikzlibrary{fit}
\usepackage{caption}
\usepackage{amsthm}

\usepackage{amsmath}

\usepackage[pagebackref=true, bookmarksopen=true,colorlinks=true, linkcolor=red,citecolor=blue]{hyperref}

\usepackage[capitalise]{cleveref}  %References will output ‘Theorem 3.21’, without you having to enter the “Theorem.”

\global\long\def\ii{\cap}%

\global\long\def\u{\cup}%

\global\long\def\I{\bigcap}%

\global\long\def\U{\bigcup}%

\global\long\def\s{\subset}%

\global\long\def\P{\prime}%

\global\long\def\ñ{\sim}%

\global\long\def\core#1{\textnormal{core}(#1)}%

\global\long\def\ker#1{\textnormal{ker}(#1)}%

\global\long\def\nucleus#1{\textnormal{nucleus}(#1)}%

\global\long\def\diadem#1{\textnormal{diadem}(#1)}%

\global\long\def\corona#1{\textnormal{corona}(#1)}%

\global\long\def\a#1{\left|#1\right|}%

\usepackage{tikz-cd}

\usepackage{tkz-graph}
\usepackage{multicol}

\usepackage{subcaption}

\newtheorem{theorem}{Theorem}[section]

\newtheorem{conjecture}[theorem]{Conjecture}
\newtheorem{corollary}[theorem]{Corollary}

\newtheorem{lemma}[theorem]{Lemma}

\newtheorem{problem}[theorem]{Problem}

\newtheorem{remark}[theorem]{Remark}

\newtheorem{observation}[theorem]{Observation}
\usepackage{comment}
\begin{document}

\begin{abstract}
	A K\H{o}nig--Egerv\'ary graph is a graph $G$ satisfying
	$\alpha(G)+\mu(G)=n(G)$, where $\alpha(G)$, $\mu(G)$, and $n(G)$ denote the
	independence number, the matching number, and the order of $G$, respectively.
	Let $\textnormal{core}(G)$ and $\textnormal{corona}(G)$ be the intersection
	and the union of all maximum independent sets of $G$.
	
	In this paper, we provide a complete characterization of graphs satisfying
	$\a{\corona G}+\a{\core G}=2\alpha(G)+1$,
	thus giving a solution to an open problem posed by Levit and Mandrescu.
	It is known that for a non-K\H{o}nig--Egerv\'ary graph with a unique odd cycle,
	the following hold:
	$\ker G=\textnormal{core}(G),\allowbreak\ 
	\left|\textnormal{corona}(G)\right|
	+\left|\textnormal{core}(G)\right|
	=2\alpha(G)+1,\allowbreak\ 
	\textnormal{corona}(G)\cup N(\textnormal{core}(G))=V(G)$.
	We extend these three results to a family of graphs containing an
	arbitrarily large number of odd cycles.
\end{abstract}

\begin{frontmatter} 
	\title{A characterization of graphs with $\a{\corona G}+\a{\core G}=2\alpha(G)+1$}

	%\author[IMASL,DEPTO]{XXXXXXXXXXXXXXXX} 	\ead{XXXXXXXXXXXX@unsl.edu.ar} 
	
	\author[IMASL,DEPTO]{Kevin Pereyra} 	\ead{kdpereyra@unsl.edu.ar}

	%\author[]{XXX} 	%\ead{xxx@XXX} 
	%\author[]{XXX} 	%\ead{xxx@XXX}

	\address[IMASL]{Instituto de Matem\'atica Aplicada San Luis, Universidad Nacional de San Luis and CONICET, San Luis, Argentina.}
	\address[DEPTO]{Departamento de Matem\'atica, Universidad Nacional de San Luis, San Luis, Argentina.} 	
	
	\date{Received: date / Accepted: date} 
	
\end{frontmatter} %

\section{Introduction}\label{sss0}

Let $\alpha(G)$ denote the cardinality of a maximum independent set,
and let $\mu(G)$ be the size of a maximum matching in $G=(V,E)$.
It is known that $\alpha(G)+\mu(G)$ equals the order of $G$,
in which case $G$ is a König--Egerváry graph 
\cite{deming1979independence,gavril1977testing,stersoul1979characterization}.
K\H{o}nig-Egerv\'{a}ry graphs have been extensively studied
\cite{bourjolly2009node,jarden2017two,levit2006alpha,levit2012critical,jaume5404413kr}.
It is known that every bipartite graph is a König--Egerváry graph 
\cite{egervary1931combinatorial}. These graphs were independently introduced by Deming \cite{deming1979independence}, Sterboul \cite{sterboul1979characterization}, and \cite{gavril1977testing}.

Let $\Omega^{*}(G)=\left\{ S:S\textnormal{ is an independent set of }G\right\}$,
$\Omega(G)=\{S:S$ is a maximum independent set of $G\}$,
$\textnormal{core}(G)=\I\left\{ S:S\in\Omega(G)\right\}$ 
\cite{levit2003alpha+}, and 
$\textnormal{corona}(G)=\U\left\{ S:S\in\Omega(G)\right\}$ 
\cite{boros2002number}. 
The number $d_{G}(X)=\left|X\right|-\left|N(X)\right|$ is the 
difference of the set $X\s V(G)$, and 
$d(G)=\max\{d_{G}(X):X\s V(G)\}$ is called the \emph{critical difference} of $G$.
A set $U\s V(G)$ is \emph{critical} if $d_{G}(U)=d(G)$ 
\cite{zhang1990finding}. 
The number $d_{I}(G)=\max\left\{ d_{G}(X):X\in\Omega^{*}(G)\right\}$ 
is called the \emph{critical independence difference} of $G$. 
If a set $X\s\Omega^{*}(G)$ satisfies $d_{G}(X)=d_{I}(G)$, then it is called 
a \emph{critical independent set} \cite{zhang1990finding}. 
Clearly, $d(G)\ge d_{I}(G)$ holds for every graph. 
It is known that $d(G)=d_{I}(G)$ for all graphs 
\cite{zhang1990finding}. 
We define $\textnormal{ker}(G)=\I\{ S:S$ is a critical independent set of $G\}$ 
\cite{levit2012vertices,lorentzen1966notes,schrijver2003combinatorial}.
Let $\textnormal{nucleus}(G)=\I\{S:S$ is a maximum critical independent 
set of $G\}$ \cite{jarden2019monotonic}, and 
$\textnormal{diadem}(G)=\U\{S:S$ is a critical independent set of $G\}$ 
\cite{short2015some}.

It is known that $\textnormal{ker}(G)\s \textnormal{core}(G)$ for every graph 
\cite{levit2012vertices}, while equality holds in bipartite graphs 
\cite{levit2013critical}, and for unicyclic non-König--Egerváry graphs 
\cite{levit2011core}.  

A graph $G$ is almost bipartite if it has only one odd cycle. In
\cite{levit2025almost} it is shown that the equality between $\ker G$
and $\textnormal{core}(G)$ also holds for every almost bipartite 
non-König--Egerváry graph. That is, we have the following:

\begin{theorem}
	[\label{levit}\cite{levit2025almost}] If $G$ is an almost bipartite
	non-König--Egerváry graph, then 
	\begin{itemize}
		\item $\ker G=\textnormal{core}(G),$
		\item $\left|\textnormal{corona}(G)\right|+\left|\textnormal{core}(G)\right|=2\alpha(G)+1,$
		\item $\textnormal{corona}(G)\cup N(\textnormal{core}(G))=V(G).$
	\end{itemize}
\end{theorem}

In particular, the second item of \cref{levit} motivates \cref{mainproblem}.

\begin{theorem}
	[\cite{levit2025almost}\label{asd12w3123}] If $G$ is an almost bipartite
	non-König--Egerváry graph, then 
	$\left|\textnormal{corona}(G)\right|+\left|\textnormal{core}(G)\right|=2\alpha(G)+1,$
\end{theorem}

\begin{problem}[\cite{levit2025almost}\label{mainproblem}]
	Characterize graphs enjoying $\left|\textnormal{corona}(G)\right|+\left|\textnormal{core}(G)\right|=2\alpha(G)+1,$
\end{problem}

In this paper, we provide a complete characterization of graphs that satisfy
\[
\a{\corona G}+\a{\core G}=2\alpha (G) +1.
\]
\noindent Using Larson’s independence decomposition \cite{larson2011critical}, our result yields a
reductive-type solution to the \cref{mainproblem}: namely, a graph satisfies the
above equality if and only if its $L^{c}$  does (\cref{120ij31}). Consequently, any structural
solution to the \cref{mainproblem} is essentially reduced to the study of
2-bicritical graphs that satisfy the equality. In addition, we extend the three items of the \cref{levit} to a family of graphs with an
arbitrarily large number of odd cycles.

The paper is organized as follows. In \cref{sss0} we present the general context of the
problem and introduce the fundamental concepts. In \cref{sss1} we fix the notation used
throughout the paper. In \cref{sss2} we obtain the characterization of the graphs
satisfying
$
\a{\corona G}+\a{\core G}=2\alpha(G)+1,
$
and in \cref{sss3} we extend the \cref{levit}. Finally, in \cref{0i12j3} we conclude the
paper by proposing several open problems and conjectures motivated by our results.

\section{Preliminaries}\label{sss1}
All graphs considered in this paper are finite, undirected, and simple. 
For any undefined terminology or notation, we refer the reader to 
Lovász and Plummer \cite{LP} or Diestel \cite{Distel}.

Let \( G = (V, E) \) be a simple graph, where \( V = V(G) \) is the finite set of vertices and \( E = E(G) \) is the set of edges, with \( E \subseteq \{\{u, v\} : u, v \in V, u \neq v\} \). We denote the edge \( e=\{u, v\} \) as \( uv \). A subgraph of \( G \) is a graph \( H \) such that \( V(H) \subseteq V(G) \) and \( E(H) \subseteq E(G) \). A subgraph \( H \) of \( G \) is called a \textit{spanning} subgraph if \( V(H) = V(G) \). 

Let \( e \in E(G) \) and \( v \in V(G) \). We define \( G - e := (V, E - \{e\}) \) and \( G - v := (V - \{v\}, \{uw \in E : u,w \neq v\}) \). If \( X \subseteq V(G) \), the \textit{induced} subgraph of \( G \) by \( X \) is the subgraph \( G[X]=(X,F) \), where \( F:=\{uv \!\in\! E(G) : u, v \!\in \! X\} \).

The number of vertices in a graph $G$ is called the \textit{order} of the graph and denoted by $\left|G\right|$ or $n(G)$.
A \textit{cycle} in $G$ is called \textit{odd} (resp. \textit{even}) if it has an odd (resp. even) number of edges.

For a vertex $v\in V(G)$, the \emph{neighborhood} of $v$ is
\[
N_G(v)=\{u\in V(G): uv\in E(G)\}.
\]
When no confusion arises, we write $N(v)$ instead of $N_G(v)$. For a set $S\subseteq V(G)$, the \emph{neighborhood} of $S$ is
\[
N_G(S)=\bigcup_{v\in S} N_G(v).
\]

\begin{comment}
	\noindent The \emph{degree} of a vertex $v\in V(G)$ is
$
\deg_G(v)=|N_G(v)|.
$
\noindent The \emph{minimum degree} of $G$ is
$
\delta(G)=\min\{\deg_G(v): v\in V(G)\}.
$
A graph $G$ is called \emph{$r$-regular} if $\deg_G(v)=r$ for every
$v\in V(G)$. A vertex $v\in V(G)$ is called an \emph{isolated vertex} if
$
\deg_G(v)=0.
$
The number of isolated vertices of a graph $G$ is denoted by $i(G)$.
\end{comment}

A \textit{matching} \(M\) in a graph \(G\) is a set of pairwise non-adjacent edges. The \textit{matching number} of \(G\), denoted by  \(\mu(G)\), is the maximum cardinality of any matching in \(G\). Matchings induce an involution on the vertex set of the graph: \(M:V(G)\rightarrow V(G)\), where \(M(v)=u\) if \(uv \in M\), and \(M(v)=v\) otherwise. If \(S, U \subseteq V(G)\) with \(S \cap U = \emptyset\), we say that \(M\) is a matching from \(S\) to \(U\) if \(M(S) \subseteq U\). A matching $M$ is \emph{perfect} if $M(v)\neq v$ for every vertex
of the graph. A matching is \emph{near-perfect} if \( \left|{v \in V(G) : M(v) = v}\right| = 1 \).  A graph is a factor-critical graph if $G-v$ has perfect matching
for every vertex $v\in V(G)$.

A vertex set \( S \subseteq V \) is \textit{independent} if, for every pair of vertices \( u, v \in S \), we have \( uv \notin E \). 
The number of vertices in a maximum independent set is denoted by \( \alpha(G) \).  A \textit{bipartite} graph is a graph whose vertex set can be partitioned into two disjoint independent sets.

\section{A reductive characterization}\label{sss2}

Before proving the main result, we need some known results.

\begin{theorem}
	[\cite{jarden2019monotonic}\label{1230ij}] $\a{\corona G}+\a{\core G}=2\alpha(G)$
if and only if $\a{\U\Gamma}+\a{\I\Gamma}=2\alpha(G)$ holds for each
non-empty $\Gamma\s\Omega(G)$. 
\end{theorem}

\begin{theorem}
	[\label{asdioj123}\cite{short2015some}] For a graph G, the following are equivalent:
\begin{itemize}
	\item $G$ is K\H{o}nig-Egerváry graph,
	\item $\diadem G=\corona G,$ and
	\item $\a{\diadem G}+\a{\nucleus G}=2\alpha(G)$. 
\end{itemize}
\end{theorem}

The \cref{asdioj123} was conjectured in \cite{jarden2019monotonic,jarden2018critical} and was finally proved in \cite{short2015some}.

\begin{theorem}\label{0i1j230i123}
	 If G is a König--Egerváry graph, then
\begin{itemize}
	\item \textnormal{\cite{levit2003alpha+}} $\corona G\u N\left(\core G\right)=V(G).$
	\item \textnormal{\cite{levit2011set}} $\a{\core G} +\a{\corona G}=2\alpha(G).$
\end{itemize}
\end{theorem}

\begin{theorem}[\label{levit1}\cite{levit2012vertices}]
	For every graph $G$, we have $\textnormal{ker}(G)\subseteq\textnormal{core}(G)$.
\end{theorem}

From \cref{1230ij} and \cref{0i1j230i123}, the following follows directly.

\begin{corollary}\label{saodi123}
	If $G$ is a K\H{o}nig-Egerváry graph, then 
\[
\a{\U\Gamma}+\a{\I\Gamma}=2\alpha(G)
\]
\noindent holds for each non-empty $\Gamma\s\Omega(G)$.
\end{corollary}

In \cite{larson2011critical}, Larson introduces the following decomposition theorem.
\begin{theorem}[\cite{larson2011critical}\label{larsonthm}]
	For any graph $G$, there is a unique set $L(G)\s V(G)$
	such that
	\begin{enumerate}
		\item $\alpha(G)=\alpha(G[L])+\alpha(G[V(G)-L(G)])$,
		\item $G[L(G)]$ is a König-Egerváry graph,
		\item for every non-empty independent set $I$ in $G[V(G)-L(G)]$, we have $\left|N(I)\right|>\left|I\right|,$
		and
		\item for every maximum critical independent set $J$ of $G$,  $L(G)=J\u N(J)$.
	\end{enumerate}
\end{theorem}

Throughout the remainder of the paper, $L(G)$ and $L^{c}(G)=V(G)-L(G)$ denote the sets
of \cref{larsonthm}; moreover, to simplify the notation, we define the induced graphs
\begin{align*}
	L_{G} & :=G[L(G)],\\
	L_{G}^{c} & :=G[L^{c}(G)].
\end{align*}

\medskip
\medskip

The notion of 2-bicritical graphs was introduced
in \cite{pulleyblank1979minimum}, and they can be characterized as follows.

\begin{theorem}[\cite{pulleyblank1979minimum}\label{1928u3123}]
	A graph $G$ is $2$-bicritical if and only if $\left|N(S)\right|>\left|S\right|$
	for every nonempty independent set $S\subseteq V(G)$.
\end{theorem}

The class of $2$-bicritical graphs can be regarded as the structural counterpart of König--Egerváry graphs \cite{larson2011critical}. In recent works, several new properties of $2$-bicritical graphs have been established; see, for instance, \cite{kevinSDKECHAR, kevinSDKEGE, kevinPOSYFACTOR}.

\medskip

\begin{observation}
	By \cref{larsonthm}, for every graph $G$ with $L^c(G)\neq \emptyset$, it follows that $L_{G}^{c}$ is a 2-bicritical graph.
\end{observation}

\medskip
\medskip

Let $\mathcal{G}$ be the family of graphs satisfying $\a{\corona G}+\a{\core G}=2\alpha(G)+1$,
that is,
\[
\mathcal{G}:=\left\{ G:\a{\corona G}+\a{\core G}=2\alpha(G)+1\right\}.
\]

The following result provides a reduction-type solution to
\cref{mainproblem}. In brief, a graph satisfies $\a{\corona G}+\a{\core G}=2\alpha(G)+1$
if and only if its $L^{c}$-part also satisfies it. Therefore, a structural
solution to \cref{mainproblem} essentially reduces to studying
2-bicritical graphs that satisfy the equality.

\medskip

\begin{theorem}\label{120ij31}
	A graph satisfies $G\in\mathcal{G}$ if and only
	if $L_{G}^{c}\in\mathcal{G}.$ 
\end{theorem}

\begin{proof}	
	Let $I$ be a maximum critical independent set. Then,
	by \cref{larsonthm}, $L(G)=I\cup N(I)$. 
	Since $\alpha(G)=\alpha(L_{G})+\alpha(L_{G}^{c})$, there exists $T\in\Omega(G)$
	such that $T\cap L(G)=I$, where $|I|=\alpha(L_{G})$ by \cref{larsonthm}.
	Then $T\cap L^{c}(G)$ can be any set from $\Omega(L_{G}^{c})$,
	in other words.
\begin{align*}
	\a{\corona G} & =\a{\U_{S\in\Omega(G)}S}\\
	& =\a{L(G)\ii\U_{S\in\Omega(G)}S}+\a{L^{c}(G)\ii\U_{S\in\Omega(G)}S}\\
	& =\a{L(G)\ii\U_{S\in\Omega(G)}S}+\a{\corona{L_{G}^{c}}}.
\end{align*}
\noindent Similarly, we show that 
\begin{align*}
	\a{\core G} & =\a{\I_{S\in\Omega(G)}S}\\
	& =\a{L(G)\ii\I_{S\in\Omega(G)}S}+\a{L^{c}(G)\ii\I_{S\in\Omega(G)}S}\\
	& =\a{L(G)\ii\I_{S\in\Omega(G)}S}+\a{\core{L_{G}^{c}}}.
\end{align*}
\noindent Therefore,
\begin{eqnarray*}
	\a{\corona G}+\a{\core G} & = & \a{\U_{S\in\Omega(G)}S}+\a{\I_{S\in\Omega(G)}S}\\
	& = & \a{L(G)\ii\U_{S\in\Omega(G)}S}+\a{\corona{L_{G}^{c}}}\\
	&  & +\a{L(G)\ii\I_{S\in\Omega(G)}S}+\a{\ker{L_{G}^{c}}}.
\end{eqnarray*}
\noindent Since $\alpha(G)=\alpha(L_{G})+\alpha(L_{G}^{c})$,
by \cref{larsonthm} $L_G$ is a Kőnig--Egerváry graph, and then by \cref{saodi123} we have
\[
\a{L(G)\ii\U_{S\in\Omega(G)}S}+\a{L(G)\ii\I_{S\in\Omega(G)}S}=2\alpha(L_{G}).
\]
\noindent Finally, we obtain that
\begin{eqnarray*}
	\a{\corona G}+\a{\core G} & = & 2\alpha(L_{G})+\a{\corona{L_{G}^{c}}}+\a{\ker{L_{G}^{c}}}.
\end{eqnarray*}

\medskip
\medskip

Now, if $G\in\mathcal{G}$, then recalling that $\alpha(G)=\alpha(L_{G})+\alpha(L_{G}^{c})$,
we obtain the following:
\begin{align*}
	2\alpha(L_{G})+\a{\corona{L_{G}^{c}}}+\a{\ker{L_{G}^{c}}} & =2\alpha(G)+1\\
	& =2\alpha(L_{G})+2\alpha(L_{G}^{c})+1.
\end{align*}
\noindent This reduces to $L_{G}^{c}\in\mathcal{G}$. 

\medskip
\medskip

Conversely, if $L_{G}^{c}\in\mathcal{G}$, then
\[
\a{\corona G}+\a{\core G}-2\alpha(L_{G})=2\alpha(L_{G}^{c})+1
\]
\noindent which again reduces to $G\in\mathcal{G}$. 
\end{proof}

\medskip
\medskip

\cref{120ij31} shows that the family $\mathcal{G}$ is completely determined by the
structure of $L^c_G$. Therefore, to understand $\mathcal{G}$ it is enough to focus on the $2$-bicritical graphs in $\mathcal{G}$,
that is, on the possible structures of $L^{c}_G$.
To this end, we recall a classical constructive description of $2$-bicritical graphs via ear--pendant decompositions.

\medskip

To analyze the possible structures of $L^c_G$, we rely on classical constructive
descriptions of 2-bicritical graphs.

We say that $G^{\P}$ is an odd homeomorph of $G$ if $G^{\P}$ is
obtained by replacing each edge of $G$ with a path of odd length,
such that these paths remain internally disjoint. Let $B$ be a subgraph
of $G$. An ear in $G$ with respect to $B$ is an odd-length path
in $G$ whose endpoints lie in $B$, but whose internal vertices
lie outside $B$, and in which all internal vertices are distinct.
A pendant in $G$ with respect to $B$ consists of an odd-length simple cycle
$C$ in $G$, vertex-disjoint from $B$, together with a positive-length
simple path with one end in $C$ and the other in $B$, with all other
vertices lying outside both $B$ and $C$. The vertex in $B$ is called
the end of the pendant.

An \emph{ear-pendant} decomposition of a graph $G$ is a sequence
$G_{0},G_{1},\dots,G_{p}=G$ of graphs where $G_{0}$ is either an
odd cycle or an odd homeomorph of $K_{4}$, and for each $i\in\{1,\dots,p\}$,
$G_{i}$ is obtained from $G_{i-1}$ by adding either an ear or a
pendant. In \cite{bourjolly1989konig} the following characterization of 2-bicritical
graphs in terms of ear-pendant decompositions was given.

\begin{theorem}[\label{puyelarpendietnedesc}\cite{bourjolly1989konig}]
	A graph is 2-bicritical if and only if it has an ear-pendant decomposition.
\end{theorem}

Using this approach, we will obtain \cref{asd12w3123} as a consequence of \cref{120ij31}. Ear--pendant decompositions provide a convenient way to isolate odd cycles in $L^c(G)$,
which will be exactly what we need in the almost bipartite setting.

\begin{proof}[Proof of \cref{asd12w3123}]
	Let $G$ be an almost bipartite non-K\H{o}nig--Egerváry graph. Then,
	by \cref{larsonthm}, $L(G)\neq\emptyset$. Hence $L_{G}^{c}$
	is a 2-bicritical graph with a unique odd cycle. By \cref{puyelarpendietnedesc},
	$L_{G}^{c}$ is exactly the odd cycle of $G$. But then,
	trivially $L_{G}^{c}\in\mathcal{G}$, and thus by \cref{120ij31} we have $G\in\mathcal{G}.$ 
\end{proof}

In the next section, using \cref{120ij31}, we will structurally extend
\cref{asd12w3123}.

\section{Graphs with many odd cycles}\label{sss3}

In this section, we show how the reductive characterization obtained in
\cref{sss2} can be exploited to extend structural properties known for almost
bipartite non-K\H{o}nig--Egerv\'ary graphs to graphs containing several odd cycles.

\medskip

Ear decompositions play a fundamental role in the structural theory of graphs
related to matching properties. An ear decomposition $G_0,G_1,\dots,G_k=G$ of a graph $G$
is a sequence of graphs with the first graph being simple (e.g., a
vertex, edge, even cycle, or odd cycle), and each graph $G_{i+1}$ 
is obtained from $G_i$ by adding an ear. Adding an ear is done
as follows: take two vertices $u$ and $v$ of $G_i$ and add a
path $P_i$ from $u$ to $v$ such that all vertices on the path
except $u$ and $v$ are new vertices.

\begin{theorem}
	[\cite{lovasz1972structure}\label{lovasz}] A graph $G$ is factor-critical
	if and only if $G$ has an odd-length ear decomposition starting from
	an odd cycle. 
\end{theorem}

The study of bipartite matching covered graphs has a long history: K\H{o}nig already
used this concept in 1915 while analyzing determinant decompositions
\cite{konig1915}, and nearly half a century later Hetyei formalized the term
\emph{elementary} instead bipartite matching covered and developed a classical characterization \cite{hetyei1964}.
It should be noted that several results concerning matching covered graphs,
such as those in \cite{LP}, are presented in the more general context of bipartite
matching covered graphs.

\medskip

In this paper, however, we will work with the following equivalent notion,
which is more convenient for our purposes. A graph $G$ is called a \emph{bipartite matching covered graph} if $G$ is
connected, bipartite, and every edge of $G$ belongs to a perfect matching.

\medskip

Bipartite matching covered graphs admit a particularly elegant structural
description in terms of ear decompositions, due to Lovász.

\begin{theorem}[\cite{lovasz1983ear}]
	 Given any bipartite matching covered graph $G$,
there exist a odd length ear decomposition
\[
G_{1},G_{2},\dots,G_{r}=G
\]
\noindent of matching covered subgraphs of $G$ where $G_{1}=K_{2}$.
\end{theorem}

In the bipartite setting, odd ear decompositions starting from $K_2$ characterize matching covered graphs.

\medskip
\medskip

A graph $G$ is called an \emph{almost-bipartite matching covered}
graph if $G$ admits an ear decomposition
\[
G_{1},G_{2},\dots,G_{r},G_{r+1}=G
\]
\noindent where
\begin{itemize}
	\item $G_{1}=K_{2}$ and
	\item $G_{i}$ is obtained from $G_{i-1}$ by adding an odd length ear for
	$i=2,\dots,r$ 
	\item $G_{r}$ is a bipartite (matching covered) graph, and
	\item $G_{r+1}=G$ is obtained from $G_{r}$ by adding an even length ear
	$P$ such that $V(P)\ii V(G_{r})=V(G_{1})$. 
\end{itemize}
\noindent See, for example, \cref{1231212s}.

\begin{figure}[H]
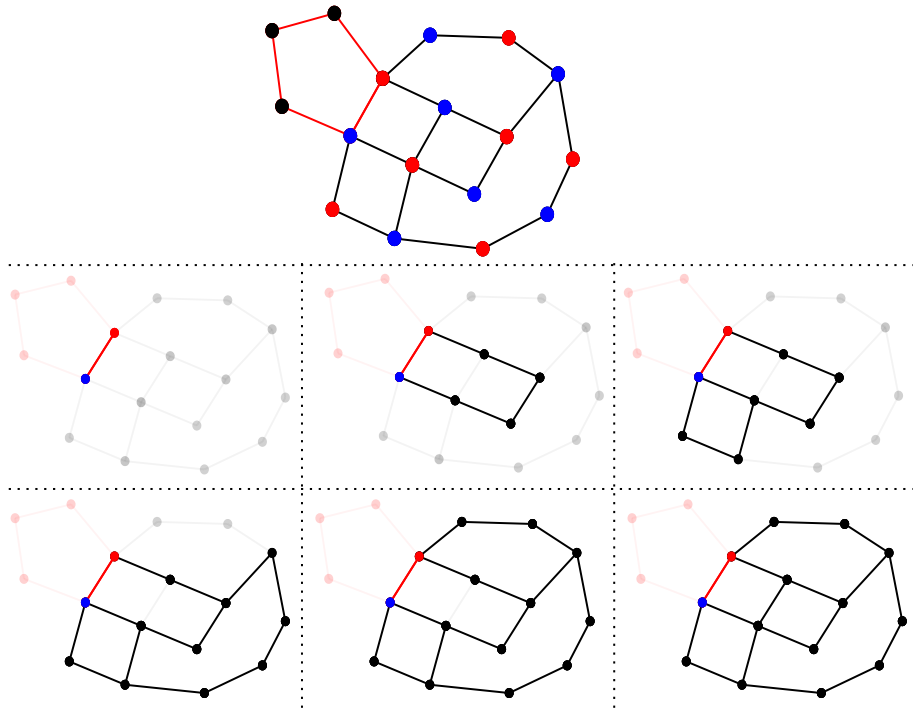

	
	\begin{center}

		\tikzset{every picture/.style={line width=0.75pt}} %set default line width to 0.75pt        
		


	\end{center}

	\caption{An example of an almost-bipartite matching covered graph.}
	\label{1231212s}
	
\end{figure}

We now present some first observations on almost-bipartite matching
covered graphs.

\begin{observation}\label{o1i23j12ji}
	Let $G$ be an almost-bipartite matching covered
	graph. Then
	\begin{itemize}
		\item $G$ is not bipartite,
		\item $G$ is a factor-critical graph,
		\item $G$ is a 2-bicritical graph,
		\item $L_{G}^{c}=G$. 
	\end{itemize}
\end{observation}

\begin{proof}
	The first item follows from the fact that the last ear in the
	ear decomposition of $G$ forms an odd cycle. The second and third
	items follow from \cref{lovasz} and \cref{puyelarpendietnedesc}, respectively, after
	rearranging the decomposition of $G$ by starting from $G_{1}$,
	then adding the last ear (the even-length ear), and subsequently adding the remaining odd ears
	in the same original order. In this way, we obtain an odd ear decomposition
	starting from an odd cycle. On the other hand, the fourth item follows
	from the third item and \cref{larsonthm}.
\end{proof}

\begin{theorem}
	[\label{19}\cite{berge2005some}] An independent set
	$S$ is maximum if and only if every independent set disjoint
	from $S$ can be matched into $S$. 
\end{theorem}

The next lemma establishes several properties of almost-bipartite
matching covered graphs that will be crucial in the sequel.

\begin{lemma}\label{asdiuh1239uh}
	Let $G$ be an almost-bipartite matching covered graph.
	Then
	\begin{itemize}
		\item $\a G=2\alpha(G)+1$,
		\item $\core G=\emptyset$,
		\item $\ker G=\emptyset$,
		\item $\corona G=V(G)$. 
	\end{itemize}
\end{lemma}

\begin{proof}
	Consider the decomposition $G_{1},G_{2},\dots,G_{r},G_{r+1}=G$
	of the almost-bipartite matching covered graph $G$ as in the definition.
	Then $G_{r}$ is a bipartite graph, say with bipartition $(A,B)$.
	Since $G_{r}$ has a perfect matching, we have $\a A=\a B$ and
	$A\in\Omega(G_{r})$. Let $P=x_{1},\dots,x_{k}$ denote the even-length ear
	that we add to $G_{r}$ to obtain $G$. Then $k$ is odd and
	$V(G_{1})=\{x_{1},x_{k}\}$. Without loss of generality,
	assume that $x_{1}\in A$ and $x_{k}\in B$. Then the set $S=$
	is independent in $G$. On the other hand, there exists a matching $M$
	from the vertices of $B$ to the vertices of $A$. Let
	$M^{\P}=M\u\{x_{3}x_{4},x_{5}x_{6},\dots,x_{k-2}x_{k-1}\}$.
	Note that any independent set disjoint from $S$ does not contain
	$x_{2}$. Therefore, $M^{\P}$ matches any independent set
	disjoint from $S$ into $S$, and by \cref{19} $S$ is a maximum
	independent set of $G$. Hence
	\[
	\alpha(G)=\a S=\a A+\frac{\a P-3}{2}=\frac{\a G-1}{2}
	=\left\lfloor \frac{\a G}{2}\right\rfloor .
	\]
	
	\noindent This implies that $\a G=2\alpha(G)+1$. Note that the sets
	\begin{eqnarray*}
		& A\u\{x_{3},x_{5},\dots,x_{k-2}\}\\
		& A\u\{x_{4},x_{6},\dots,x_{k-1}\}\\
		& B\u\{x_{3},x_{5},\dots,x_{k-2}\}\\
		& B\u\{x_{2},x_{4},\dots,x_{k-3}\}
	\end{eqnarray*}
	
	\noindent are all maximum independent sets of $G$. Hence $\core G=\emptyset$
	and $\corona G=V(G)$. On the other hand, $\ker G=\emptyset$ as a consequence
	of item~3 of \cref{o1i23j12ji} and \cref{1928u3123}. 
\end{proof}

\begin{corollary}
	Every almost-bipartite matching covered graph is
	in $\mathcal{G}.$
\end{corollary}

Thus, as a consequence of \cref{120ij31}, we obtain the following.

\begin{theorem}\label{a9uih123}
	Let $G$ be a graph such that $L_{G}^{c}$ is an almost-bipartite
	matching covered graph. Then $G\in\mathcal{G}$. 
\end{theorem}

Of course, an almost bipartite non-K\H{o}nig--Egerváry graph $G$
satisfies the hypothesis of \cref{a9uih123}, and hence \cref{a9uih123}
generalizes \cref{asd12w3123}.

$ $

Unlike almost bipartite non-K\H{o}nig--Egerváry graphs,
a graph $G$ such that $L_{G}^{c}$ is an almost-bipartite matching
covered graph does not necessarily satisfy $\ker G=\core G$; for example,
see \cref{123129iu3}. However, the equality holds when $L(G)=\emptyset$,
as a consequence of \cref{asdiuh1239uh}, which extends the family of graphs
known to satisfy $\ker G=\core G.$

\begin{figure}[H]
	
	\begin{center}

		\tikzset{every picture/.style={line width=0.75pt}} %set default line width to 0.75pt        
		
		\begin{tikzpicture}[x=0.75pt,y=0.75pt,yscale=-1,xscale=1]
			
			%Straight Lines [id:da5586130535428827] 
			\draw    (231.61,157.91) -- (292.84,157.91) ;
			\draw [shift={(292.84,157.91)}, rotate = 0] [color={rgb, 255:red, 0; green, 0; blue, 0 }  ][fill={rgb, 255:red, 0; green, 0; blue, 0 }  ][line width=0.75]      (0, 0) circle [x radius= 3.35, y radius= 3.35]   ;
			\draw [shift={(231.61,157.91)}, rotate = 0] [color={rgb, 255:red, 0; green, 0; blue, 0 }  ][fill={rgb, 255:red, 0; green, 0; blue, 0 }  ][line width=0.75]      (0, 0) circle [x radius= 3.35, y radius= 3.35]   ;
			%Straight Lines [id:da4328135513600557] 
			\draw    (292.84,157.91) -- (184.93,225.57) ;
			\draw [shift={(184.93,225.57)}, rotate = 147.91] [color={rgb, 255:red, 0; green, 0; blue, 0 }  ][fill={rgb, 255:red, 0; green, 0; blue, 0 }  ][line width=0.75]      (0, 0) circle [x radius= 3.35, y radius= 3.35]   ;
			\draw [shift={(292.84,157.91)}, rotate = 147.91] [color={rgb, 255:red, 0; green, 0; blue, 0 }  ][fill={rgb, 255:red, 0; green, 0; blue, 0 }  ][line width=0.75]      (0, 0) circle [x radius= 3.35, y radius= 3.35]   ;
			%Straight Lines [id:da34652301365046956] 
			\draw    (292.84,157.91) -- (416.43,157.91) ;
			\draw [shift={(416.43,157.91)}, rotate = 0] [color={rgb, 255:red, 0; green, 0; blue, 0 }  ][fill={rgb, 255:red, 0; green, 0; blue, 0 }  ][line width=0.75]      (0, 0) circle [x radius= 3.35, y radius= 3.35]   ;
			\draw [shift={(292.84,157.91)}, rotate = 0] [color={rgb, 255:red, 0; green, 0; blue, 0 }  ][fill={rgb, 255:red, 0; green, 0; blue, 0 }  ][line width=0.75]      (0, 0) circle [x radius= 3.35, y radius= 3.35]   ;
			%Straight Lines [id:da5137806786194634] 
			\draw    (184.93,225.57) -- (184.93,90.24) ;
			\draw [shift={(184.93,90.24)}, rotate = 270] [color={rgb, 255:red, 0; green, 0; blue, 0 }  ][fill={rgb, 255:red, 0; green, 0; blue, 0 }  ][line width=0.75]      (0, 0) circle [x radius= 3.35, y radius= 3.35]   ;
			\draw [shift={(184.93,225.57)}, rotate = 270] [color={rgb, 255:red, 0; green, 0; blue, 0 }  ][fill={rgb, 255:red, 0; green, 0; blue, 0 }  ][line width=0.75]      (0, 0) circle [x radius= 3.35, y radius= 3.35]   ;
			%Straight Lines [id:da7464629878136677] 
			\draw    (184.93,225.57) -- (231.61,157.91) ;
			\draw [shift={(231.61,157.91)}, rotate = 304.6] [color={rgb, 255:red, 0; green, 0; blue, 0 }  ][fill={rgb, 255:red, 0; green, 0; blue, 0 }  ][line width=0.75]      (0, 0) circle [x radius= 3.35, y radius= 3.35]   ;
			\draw [shift={(184.93,225.57)}, rotate = 304.6] [color={rgb, 255:red, 0; green, 0; blue, 0 }  ][fill={rgb, 255:red, 0; green, 0; blue, 0 }  ][line width=0.75]      (0, 0) circle [x radius= 3.35, y radius= 3.35]   ;
			%Straight Lines [id:da5989819345012146] 
			\draw    (184.93,90.24) -- (231.61,157.91) ;
			\draw [shift={(231.61,157.91)}, rotate = 55.4] [color={rgb, 255:red, 0; green, 0; blue, 0 }  ][fill={rgb, 255:red, 0; green, 0; blue, 0 }  ][line width=0.75]      (0, 0) circle [x radius= 3.35, y radius= 3.35]   ;
			\draw [shift={(184.93,90.24)}, rotate = 55.4] [color={rgb, 255:red, 0; green, 0; blue, 0 }  ][fill={rgb, 255:red, 0; green, 0; blue, 0 }  ][line width=0.75]      (0, 0) circle [x radius= 3.35, y radius= 3.35]   ;
			%Straight Lines [id:da8658384065662791] 
			\draw    (292.84,157.91) -- (184.93,90.24) ;
			\draw [shift={(184.93,90.24)}, rotate = 212.09] [color={rgb, 255:red, 0; green, 0; blue, 0 }  ][fill={rgb, 255:red, 0; green, 0; blue, 0 }  ][line width=0.75]      (0, 0) circle [x radius= 3.35, y radius= 3.35]   ;
			\draw [shift={(292.84,157.91)}, rotate = 212.09] [color={rgb, 255:red, 0; green, 0; blue, 0 }  ][fill={rgb, 255:red, 0; green, 0; blue, 0 }  ][line width=0.75]      (0, 0) circle [x radius= 3.35, y radius= 3.35]   ;

		\end{tikzpicture}

	\end{center}

\caption{An example of a graph $G$ such that $L_{G}^{c}$ is an almost-bipartite matching covered graph, but $\ker G\neq\core G.$}

	\label{123129iu3}
	
\end{figure}
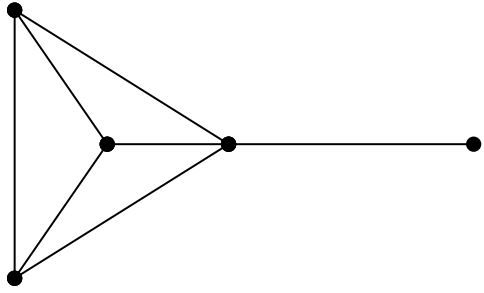

\begin{remark}
	If $G$ is an almost-bipartite matching covered graph,
	then $\ker G=\core G =\emptyset.$ 
\end{remark}

To conclude, in \cref{asdopijh123} we show another fundamental property of almost bipartite
non-K\H{o}nig--Egerváry graphs, which is inherited by graphs for which
$L_{G}^{c}$ is an almost-bipartite matching covered graph. Before that,
we need the following facts.

\begin{lemma}\label{asdouih12}
	Let $G$ be a K\H{o}nig--Egerváry graph and let $\emptyset\neq\Gamma\subseteq\Omega(G)$.
	Then there exists a matching from $V(G)-\U\Gamma$ into $\I\Gamma$.
\end{lemma}
\begin{proof}
	Suppose that $\a{\Gamma}>1$, otherwise there is nothing to prove.
	Let $I\in\Gamma$. Since $G$ is a K\H{o}nig--Egerváry graph, there exists a
	maximum matching $M$ of $G$ that matches all vertices of $V(G)-I$ into
	$I$. We show that $M$ matches $V(G)-\U\Gamma$ into $\I\Gamma$.
	Let $v\in V(G)-\U\Gamma$ and let $I^{\P}\in\Gamma-\{I\}$. Note that
	$M$ matches all vertices of $(V(G)-I)\cap I^{\P}$ into $I-I^{\P}$;
	therefore $M(v)\in I\cap I^{\P}$. Since $I^{\P}$ is arbitrary in
	$\Gamma-\{I\}$, we obtain that
	\[
	M(v)\in\I\Gamma.
	\]
	\noindent Thus $M$ matches $V(G)-\U\Gamma$ into $\I\Gamma$,
	as desired.
\end{proof}

In \cite{jarden2019monotonic} it is shown that \cref{asdouih12} holds when $\Gamma=\Omega(G)$.

\begin{theorem}\label{asdouih129i3}
	Let $G$ be a K\H{o}nig--Egerváry graph and let $\emptyset\neq\Gamma\subseteq\Omega(G)$.
	Then
	\[
	\U\Gamma\cup N\left(\I\Gamma\right)=V(G).
	\]
\end{theorem}
\begin{proof}
	By \cref{asdouih12}, $V(G)-\U\Gamma\subseteq N\left(\I\Gamma\right)$. Hence
	\[
	V(G)=\U\Gamma\cup\left(V(G)-\U\Gamma\right)\subseteq\U\Gamma\cup N\left(\I\Gamma\right),
	\]
	\noindent which proves that $\U\Gamma\cup N\left(\I\Gamma\right)=V(G)$.
\end{proof}

\begin{theorem}\label{asdopijh123}
	Let $G$ be a graph such that $L_{G}^{c}$ is an almost-bipartite
	matching covered graph. Then
	\[
	\corona G\cup N(\core G)=V(G).
	\]
\end{theorem}

\begin{proof}
	Let $I$ be a maximum critical independent set. Then,
	by \cref{larsonthm}, $L(G)=I\cup N(I)$. 
	
	Since $\alpha(G)=\alpha(L_{G})+\alpha(L_{G}^{c})$, there exists $T\in\Omega(G)$
	such that $T\cap L(G)=I$, where $|I|=\alpha(L_{G})$, by \cref{larsonthm}.
	Then $T\cap L^{c}(G)$ can be any set from $\Omega(L_{G}^{c})$,
	that is,
	\begin{align*}
		\corona G & =\U_{S\in\Omega(G)}S\\
		& =\left(L(G)\cap\U_{S\in\Omega(G)}S\right)\cup\left(L^{c}(G)\cap\U_{S\in\Omega(G)}S\right)\\
		& =\left(L(G)\cap\U_{S\in\Omega(G)}S\right)\cup\corona{L_{G}^{c}}.
	\end{align*}
	\noindent Similarly, we show that 
	\begin{align*}
		\core G & =L(G)\cap\I_{S\in\Omega(G)}S.
	\end{align*}
	\noindent Therefore, by \cref{asdiuh1239uh}, we obtain
	\begin{align*}
		\core G & =\left(L(G)\cap\I_{S\in\Omega(G)}S\right)\cup\core{L_{G}^{c}},\\
		\corona G & =\left(L(G)\cap\U_{S\in\Omega(G)}S\right)\cup L^{c}(G).
	\end{align*}
	\noindent But by \cref{larsonthm}, $L_{G}$ is a K\H{o}nig--Egerváry graph.
	Then, by \cref{asdouih129i3}, we have
	\[	
	\left(L(G)\cap\U_{S\in\Omega(G)}S\right)\cup N\left(L(G)\cap\I_{S\in\Omega(G)}S\right)=L(G).
	\]
	\noindent Moreover, $L_{G}^{c}$ has no isolated vertices, and hence
	\[
	L(G)\cup N(L^{c}(G))=V(G)\subseteq\corona G\cup N\left(\core G\right).
	\]
	\noindent This completes the proof.
\end{proof}

\section{Conclusions}\label{0i12j3}

Several classes of graphs satisfying
\[
\a{\corona G}+\a{\core G}=2\alpha(G)+k
\]
have been studied in
\cite{kevin2025RDG,kevinBAB,kevinpastine}.
In particular, the following results are known.

\begin{theorem}[\cite{kevinpastine}]
	For every graph $G$,
	\[
	\left|\textnormal{corona}(G)\right|
	+\left|\textnormal{core}(G)\right|
	\le 2\alpha(G)+k,
	\]
	where $k$ is the number of odd cycles in $G$.
\end{theorem}

\begin{theorem}[\cite{kevin2025RDG}]
	If $G$ is an $R$-disjoint graph with exactly $k$ disjoint odd cycles, then
	\[
	\left|\textnormal{corona}(G)\right|
	+\left|\textnormal{core}(G)\right|
	= 2\alpha(G)+k.
	\]
\end{theorem}

\begin{theorem}[\cite{kevinBAB}]
	Let $G=(B,G_{1},\dots,G_{k})$ be a BAB-graph. Then
	\[
	\left|\textnormal{ker}(G)\right|
	+\left|\textnormal{corona}(G)\right|
	\le 2\alpha(G)+k
	= d(G)+\left|G\right|.
	\]
\end{theorem}

The results of this paper show that the study of graphs satisfying
\[
|\corona G| + |\core G| = 2\alpha(G) + 1
\]
can be reduced, via Larson’s independence decomposition, to the study of
2-bicritical graphs satisfying the same equality. The results obtained in this paper naturally motivate the following open problems.

\begin{problem}
	Characterize the graphs satisfying
	\[
	\a{\corona G}+\a{\core G}=2\alpha(G)+k,
	\]
	for $k\ge 2$.
\end{problem}

\begin{problem}
	Characterize the $2$-bicritical graphs belonging to $\mathcal{G}$.
\end{problem}

\begin{conjecture}
	For every graph $G$,
	$
	\a{\corona G}+\a{\core G}=2\alpha(G)+k
	$
	if and only if
	$
	\a{\corona{L_{G}^{c}}}
	+\a{\core{L_{G}^{c}}}
	=2\alpha(L_{G}^{c})+k.
	$
\end{conjecture}

\medskip

Precise definitions of the notions involved can be found in the cited papers.

\section*{Acknowledgments}

	This work was partially supported by Universidad Nacional de San Luis, grants PROICO 03-0723 and PROIPRO 03-2923, MATH AmSud, grant 22-MATH-02, Consejo Nacional de Investigaciones
	Cient\'ificas y T\'ecnicas grant PIP 11220220100068CO and Agencia I+D+I grants PICT 2020-00549 and PICT 2020-04064.

	\section*{Declaration of generative AI and AI-assisted technologies in the writing process}
	During the preparation of this work the authors used ChatGPT-3.5 in order to improve the grammar of several paragraphs of the text. After using this service, the authors reviewed and edited the content as needed and take full responsibility for the content of the publication.

\section*{Data availability}

Data sharing not applicable to this article as no datasets were generated or analyzed during the current study.

\section*{Declarations}

\noindent\textbf{Conflict of interest} \ The authors declare that they have no conflict of interest.

\bibliographystyle{apalike}

\bibliography{TAGcitasV2025}

\end{document}